\documentclass[12pt]{article}
\usepackage[latin1]{inputenc}
\usepackage{cmap}
\usepackage{lmodern}

\usepackage{amssymb, amsmath, amsthm}
\usepackage[a4paper,top=25mm,bottom=25mm,left=25mm,right=25mm]{geometry}
\usepackage{etex}
\usepackage{ragged2e}

\usepackage{authblk} 
\usepackage{pifont}
\usepackage{graphicx}
\usepackage[usenames,dvipsnames,svgnames,table]{xcolor}
\usepackage[figuresright]{rotating}
\usepackage{xtab} 
\usepackage{longtable} 
\usepackage{multirow}
\usepackage{footnote}
\usepackage[stable]{footmisc}
\usepackage{chngpage} 
\usepackage{pdflscape} 
\usepackage{tocbibind} 

\usepackage{pgfplots}
\pgfplotsset{compat=1.14}
\usepackage{setspace}

\usepackage{array}
\newcolumntype{K}[1]{>{\centering\arraybackslash$}p{#1}<{$}}

\makesavenoteenv{tabular}
\usepackage{tabularx}
\usepackage{booktabs}
\usepackage{threeparttable}
\usepackage[referable]{threeparttablex} 
\newcolumntype{R}{>{\raggedleft\arraybackslash}X}
\newcolumntype{L}{>{\raggedright\arraybackslash}X}
\newcolumntype{C}{>{\centering\arraybackslash}X}
\newcolumntype{A}{>{\columncolor{gray!25}}C}
\newcolumntype{a}{>{\columncolor{gray!25}}c}

\newlength{\tablen}

\usepackage{dcolumn} 
\newcolumntype{.}{D{.}{.}{-1}}

\usepackage{tikz}
\usetikzlibrary{arrows,shapes}
\usepackage[semicolon]{natbib}
\usepackage[hyphens]{url}
\usepackage{hyperref} 
\hypersetup{
  colorlinks   = true,    
  urlcolor     = blue,    
  linkcolor    = blue,    
  citecolor    = red      
}
\usepackage{microtype}
\usepackage[justification=centerfirst]{caption} 

\usepackage[labelformat=simple]{subcaption}

\DeclareCaptionLabelFormat{parenthesis}{(#2)}
\makeatletter
\renewcommand\p@subfigure{\arabic{figure}.}
\makeatother

\DeclareCaptionLabelFormat{parenthesis}{(#2)}
\makeatletter
\renewcommand\p@subtable{\arabic{table}.}
\makeatother

\usepackage{enumitem}

\setlist[itemize]{leftmargin=2.5\parindent}
\setlist[enumerate]{leftmargin=2.5\parindent}

\theoremstyle{plain}

\newtheorem{corollary}{Corollary}[section]
\newtheorem{lemma}{Lemma}[section]
\newtheorem{proposition}{Proposition}[section]
\newtheorem{theorem}{Theorem}[section]

\theoremstyle{definition}
\newtheorem{axiom}{Axiom}[section]

\newtheorem{definition}{Definition}[section]
\newtheorem{example}{Example}[section]

\theoremstyle{remark}

\newtheorem{remark}{Remark}[section]

\def\keywords{\vspace{.5em} 
{\noindent \textit{Keywords}:\,}}

\def\JEL{\vspace{.5em} 
{\noindent \textbf{\emph{JEL} classification number}:\,}}

\def\AMS{\vspace{.5em} 
{\noindent \textbf{\emph{MSC} class}:\,}}

\author{\href{https://sites.google.com/site/laszlocsato87/}{L\'aszl\'o Csat\'o}\thanks{~e-mail: laszlo.csato@uni-corvinus.hu} }
\affil{Institute for Computer Science and Control, Hungarian Academy of Sciences (MTA SZTAKI) \\
Laboratory on Engineering and Management Intelligence, Research Group of Operations Research and Decision Systems}
\affil{Corvinus University of Budapest (BCE) \\
Department of Operations Research and Actuarial Sciences}
\affil{Budapest, Hungary}
\title{Characterization of the row geometric mean ranking with a group consensus axiom}
\date{\today}

\begin{document}

\maketitle

\begin{abstract}
An axiomatic approach is applied to the problem of extracting a ranking of the alternatives from a pairwise comparison ratio matrix.
The ordering induced by row geometric mean method is proved to be uniquely determined by three independent axioms, anonymity (independence of the labelling of alternatives), responsiveness (a kind of monotonicity property) and aggregation invariance, which requires the preservation of group consensus, that is, the pairwise ranking between two alternatives should remain unchanged if unanimous individual preferences are combined by geometric mean.


\JEL{C44, D71}

\AMS{90B50, 91B08}
\end{abstract}

\section{Introduction} \label{Sec1}


Preferences of decision makers are often represented by pairwise comparisons when numerical answers to questions like 'How many times alternative $i$ is better than alternative $j$?' are collected into a positive reciprocal matrix \citep{Saaty1980}.
The basic issue in this field is to derive weights from a given set of comparisons, which can be used for measuring the importance of certain decision options, or for determining a \emph{ranking} of the alternatives.
Naturally, the assignment of weights is not necessarily based on the pairwise comparisons paradigm (see, e.g., \citet{JanickiSoudkhah2014}), even though we will use this approach throughout the paper.

Since one can choose among a plethora of weighting methods, an axiomatic approach is worth to consider for this purpose. Introduction and justification of reasonable properties may reveal the advantages and disadvantages of certain procedures, and the axioms may even \emph{characterize}, uniquely determine the weights.

Probably the first work on this topic, \citet{Fichtner1984} characterized the \emph{row geometric mean} -- sometimes called logarithmic least squares -- \emph{method} \citep{Rabinowitz1976, CrawfordWilliams1980, CrawfordWilliams1985, DeGraan1980} by using four axioms, correctness in the consistent case, comparison order invariance, smoothness and power invariance.
Furthermore, the \emph{eigenvector method} \citep{Saaty1980} is uniquely determined by correctness in the consistent case, comparison order invariance, smoothness and rank preservation \citep{Fichtner1986}, that is, it can be obtained with changing only one property in the previous result.

From this set of axioms, correctness in the consistent case and comparison order invariance are almost impossible to debate.
Nevertheless, there exists a goal-programming method satisfying power invariance and a slightly modified version of smoothness besides these two basic axioms, which possesses the additional property that the presence of a single outlier cannot prevent the identification of the correct priority vector \citep{Bryson1995}.
\citet{CookKress1988} approached the problem by focusing on distance measures in order to get another goal programming method on an axiomatic basis.

Smoothness and power invariance can be entirely left out from the characterization of the row geometric mean method. \citet{BarzilaiCookGolany1987} substitute them with a consistency-like axiom by introducing two procedures which are required to result in the same preference vector: (1) some pairwise comparison matrices are aggregated to one matrix and the solution is computed for this matrix, (2) the priorities are derived separately for each matrix and combined by the geometric mean. We think it is not a simple condition immediately to adopt.
\citet{Barzilai1997} replaced this axiom and comparison order invariance with essentially demanding that each individual weight is a function of the entries in the corresponding row of the pairwise comparison matrix only. Joining to \citet{Dijkstra2013}, we are also somewhat uncomfortable with this premise.
\citet{Csato2018c} characterizes row geometric mean by assuming the weight vector to be independent of an arbitrary multiplication of matrix elements along a 3-cycle by a positive scalar.

To conclude, the problem of weight derivation seems to be not finally settled by previous axiomatizations.
Therefore we want to provide a characterization of the row geometric mean ranking from the perspective of group decision making.

Focusing on the ranking is a departure from the existing literature, which requires some explanation.
First, similarly to the case of inconsistency indices \citep{Csato2018e, Csato2018a}, our setting probably makes the axioms more easy to motivate and the result to understand.
Second, weighting methods are often used only to determine a ranking of the alternatives \citep{SaatyHu1998}.
Third, the main result essentially depends on an axiom called \emph{aggregation invariance} \citep{Csato2017b}, that is, the pairwise ranking between two alternatives should remain unchanged if unanimous individual preferences are combined by geometric mean. According to our knowledge, this property does not have an equivalent form for ratings, while similar conditions have been extensively used in social choice theory \citep{Young1974, NitzanRubinstein1981, ChebotarevShamis1998a, vandenBrinkGilles2009, Gonzalez-DiazHendrickxLohmann2013, Csato2018f, Csato2018g}.
Furthermore, since the exact meaning of aggregation invariance is determined by the aggregation procedure of pairwise comparison matrices, our axiomatization practically follows from the central work of \citet{AczelSaaty1983} on synthesizing ratio judgements.

Other axioms used in the characterization are relatively straightforward: \emph{anonymity} is probably the most natural independence property, while \emph{responsiveness} is a standard monotonicity condition directly implied by the representation of decision makers' preferences.

Presenting an axiomatic characterization does not mean that we accept all properties involved as wholly justified and unquestionable. However, \emph{if} one agrees with our axioms, then geometric mean remains the only choice. Consequently, using any other method requires explaining the violation of at least one axiom.

The paper has the following structure.
Section~\ref{Sec2} defines pairwise comparison matrices, weighting and ranking methods.
Section~\ref{Sec3} recalls the axioms introduced in \citet{Csato2017b}, and presents three new properties.
Some connections among these requirements are revealed in Section~\ref{Sec4}.
Section~\ref{Sec5} analyses the rankings induced by the eigenvector and row geometric mean methods with respect to the axioms.
Section~\ref{Sec6} provides the main result, a characterization of the row geometric mean ranking.
Finally, our contributions are summarized in Section~\ref{Sec7}.

\section{Preliminaries} \label{Sec2}

Let $\mathbb{R}^{n}_+$ and $\mathbb{R}^{n \times n}_+$ be the set of positive (with all elements greater than zero) vectors of size $n$ and matrices of size $n \times n$, respectively.
Let $N = \{ 1,2, \dots ,n \}$ be the set of alternatives.

\begin{definition} \label{Def21}
\emph{Pairwise comparison matrix}:
Matrix $\mathbf{A} = \left[ a_{ij} \right] \in \mathbb{R}^{n \times n}_+$ is a \emph{pairwise comparison matrix} if $a_{ji} = 1/a_{ij}$ for all $1 \leq i,j \leq n$.
\end{definition}

The set of pairwise comparison matrices of size $n \times n$ is denoted by $\mathcal{A}^{n \times n}$.

Let $\mathbf{1} \in \mathcal{A}^{n \times n}$ be the pairwise comparison matrix with all of its elements equal to $1$.

A pairwise comparison matrix $\mathbf{A} \in \mathcal{A}^{n \times n}$ is said to be \emph{consistent} if $a_{ik} = a_{ij} a_{jk}$ for all $1 \leq i,j,k \leq n$.
Otherwise, it is \emph{inconsistent}.

\begin{definition} \label{Def22}
\emph{Weight vector}:
Vector $\mathbf{w}  = \left[ w_{i} \right] \in \mathbb{R}^n_+$ is a \emph{weight vector} if $\sum_{i=1}^n w_{i} = 1$.
\end{definition}

The set of weight vectors of size $n$ is denoted by $\mathcal{R}^{n}$.

\begin{definition} \label{Def23}
\emph{Weighting method}:
Mapping $f: \mathcal{A}^{n \times n} \to \mathcal{R}^{n}$ is a \emph{weighting method}.
\end{definition}

A weighting method assigns a weight vector to every pairwise comparison matrix.

Several weighting methods have been suggested in the literature, see \citet{ChooWedley2004} for an overview. 
We discuss only two of them, which are probably the most popular.
The eigenvector method has also been addressed in the paper that introduced the crucial axiom of our characterization \citep{Csato2017b}.
The row geometric mean method is in the focus of the discussion.

\begin{definition} \label{Def24}
\emph{Eigenvector method} ($EM$) \citep{Saaty1980}:
The \emph{eigenvector method} is the mapping $\mathbf{A} \to \mathbf{w}^{EM} (\mathbf{A})$ such that
\[
\mathbf{A} \mathbf{w}^{EM}(\mathbf{A}) = \lambda_{\max} \mathbf{w}^{EM}(\mathbf{A}),
\]
where $\lambda_{\max}$ denotes the maximal eigenvalue, also known as principal or Perron eigenvalue, of (positive) matrix $\mathbf{A}$.
\end{definition}


\begin{definition} \label{Def25}
\emph{Row geometric mean method} ($RGM$) \citep{Rabinowitz1976, CrawfordWilliams1980, CrawfordWilliams1985, DeGraan1980}: 
The \emph{row geometric mean method} is the mapping $\mathbf{A} \to \mathbf{w}^{RGM} (\mathbf{A})$ such that the weight vector $\mathbf{w}^{RGM} (\mathbf{A})$ is the unique solution of the following optimization problem:
\begin{equation} \label{Eq_LLSM}
\min_{\mathbf{w} \in \mathcal{R}^n} \sum_{i=1}^n \sum_{j=1}^n \left[ \log a_{ij} - \log \left( \frac{w_i}{w_j} \right) \right]^2.
\end{equation}
\end{definition}

$RGM$ seeks for a weight vector that generates the consistent pairwise comparison matrix which is the closest to the pairwise comparison matrix $\mathbf{A}$ if the metric of logarithmic least squares is applied. Therefore it is often called the \emph{logarithmic least squares method}.

The name row geometric mean originates from the formula of the solution to \eqref{Eq_LLSM}, which is
\[
w_i^{RGM}(\mathbf{A}) = \frac{\prod_{j=1}^n a_{ij}^{1/n}}{\sum_{k=1}^n \prod_{j=1}^n a_{kj}^{1/n}}.
\]

Weighting methods are often used to derive a \emph{ranking} of the alternatives.
Ranking $\succeq$ is a weak order on the set of alternatives $N$, so it is complete (for all $i,j \in N$: $i \succeq j$ or $i \preceq j$) and transitive (for all $i,j,k \in N$: $i \succeq j$ and $j \succeq k$ implies $i \succeq k$).

The asymmetric and symmetric parts of a ranking $\succeq$ will be written as $\succ$ and $\sim$, respectively: $i \succ j$ if and only if $i \succeq j$ and not $j \succeq i$, $i \sim j$ if and only if $i \succeq j$ and $j \succeq i$.

The set of possible rankings on $n$ alternatives is denoted by $\mathfrak{R}^{n}$.

\begin{definition} \label{Def26}
\emph{Ranking method}:
Mapping $g: \mathcal{A}^{n \times n} \to \mathfrak{R}^{n}$ is a \emph{ranking method}.
\end{definition}

A ranking method assigns a ranking of the alternatives to every pairwise comparison matrix. We use the convention that $\succeq_{\mathbf{A}}^g$ is the ranking assigned by ranking method $g$ for pairwise comparison matrix $\mathbf{A} \in \mathcal{A}^{n \times n}$.

All weighting methods induce a ranking method, for instance:
\begin{itemize}
\item
the \emph{eigenvector ranking method} is denoted by $\succeq^{EM}$, where $i \succeq^{EM}_\mathbf{A} j$ if and only if $w_i^{EM}(\mathbf{A}) \geq w_j^{EM}(\mathbf{A})$;
\item
the \emph{row geometric mean ranking method} is denoted by $\succeq^{RGM}$, where $i \succeq^{RGM}_\mathbf{A} j$ if and only if $w_i^{RGM}(\mathbf{A}) \geq w_j^{RGM}(\mathbf{A})$.
\end{itemize}

\section{Axioms} \label{Sec3}

The six properties discussed here concern ranking methods, that is, they only deal with the relative importance of alternatives.
Some earlier works have used similar axioms for rankings.
\citet{SaatyVargas1984a} introduce the properties strong and weak rank preservation.
\citet{GenestLapointeDrury1993} examine the effect of a coding parameter for ordinal preferences on the ordering of alternatives from $EM$.
\citet{CsatoRonyai2016} discuss a condition on the ranking of alternatives derived from an incomplete pairwise comparison matrix.
\citet{PerezMokotoff2016} show an example of strong rank reversal in group decision making by $EM$.

First, let us briefly recall three axioms from \citet{Csato2017b}.

\begin{axiom} \label{Axiom1}
\emph{Anonymity} ($ANO$):
Let $\mathbf{A} = \left[ a_{ij} \right] \in \mathcal{A}^{n \times n}$ be a pairwise comparison matrix, $\sigma: N \rightarrow N$ be a permutation on the set of alternatives, and $\sigma(\mathbf{A}) = \left[ \sigma(a)_{ij} \right] \in \mathcal{A}^{n \times n}$ be the pairwise comparison matrix obtained from $\mathbf{A}$ by this permutation such that $\sigma(a)_{ij} = a_{\sigma(i) \sigma(j)}$.
Ranking method $g: \mathcal{A}^{n \times n} \to \mathfrak{R}^n$ is \emph{anonymous} if $i \succeq^g_\mathbf{A} j \iff \sigma(i) \succeq^g_{\sigma(\mathbf{A})} \sigma(j)$ for all $1 \leq i,j \leq n$.
\end{axiom}

$ANO$ demands the ranking of alternatives to be independent of their labels, which is important because the 'names' of the alternatives can be arbitrary.
This property was used under the name \emph{comparison order invariance} by \citet{Fichtner1984} for weighting methods.

\begin{definition} \label{Def31}
\emph{Aggregation of pairwise comparison matrices}:
Let $\mathbf{A}^{(1)} = \left[ a_{ij}^{(1)} \right] \in \mathcal{A}^{n \times n}$, $\mathbf{A}^{(2)} = \left[ a_{ij}^{(2)} \right] \in \mathcal{A}^{n \times n}$, $\dots$, $\mathbf{A}^{(k)} = \left[ a_{ij}^{(k)} \right] \in \mathcal{A}^{n \times n}$ be any pairwise comparison matrices. Their \emph{aggregate} is the pairwise comparison matrix $\mathbf{A}^{(1)} \oplus \mathbf{A}^{(2)} \oplus \dots \oplus \mathbf{A}^{(k)} = \left[ \sqrt[k]{a_{ij}^{(1)} a_{ij}^{(2)} \cdots a_{ij}^{(k)}} \right] \in \mathcal{A}^{n \times n}$.
\end{definition}

Aggregation is equivalent to taking the geometric mean of all corresponding matrix elements.
\citet{AczelSaaty1983} show geometric mean to be the only reasonable aggregation procedure, the unique quasiarithmetic mean satisfying reciprocity and positive homogeneity. According to reciprocity, the aggregated matrix is a pairwise comparison matrix, too, while positive homogeneity means that multiplying all individual preferences by the same positive scalar leads to an appropriate change in the aggregated preferences. 


\begin{axiom} \label{Axiom2}
\emph{Aggregation invariance} ($AI$):
Let $\mathbf{A}^{(1)},\mathbf{A}^{(2)}, \dots , \mathbf{A}^{(k)} \in \mathcal{A}^{n \times n}$ be any pairwise comparison matrices. Let $g: \mathcal{A}^{n \times n} \to \mathfrak{R}^n$ be a ranking method such that $i \succeq^g_{\mathbf{A}^{(\ell)}} j$ for all $1 \leq \ell \leq k$.
$g$ is called \emph{aggregation invariant} if $i \succeq^g_{\mathbf{A}^{(1)} \oplus \mathbf{A}^{(2)} \oplus \dots \oplus \mathbf{A}^{(k)}} j$, furthermore, $i \succ^g_{\mathbf{A}^{(1)} \oplus \mathbf{A}^{(2)} \oplus \dots \oplus \mathbf{A}^{(k)}} j$ if $i \succ^g_{\mathbf{A}^{(\ell)}} j$ for at least  one $1 \leq \ell \leq k$.
\end{axiom}

$AI$ is an intuitive condition of group decision making: if individuals unanimously agree that alternative $i$ is not worse than $j$, this relation should be preserved when their opinions are combined, i.e., it should also be reflected by the collective preferences.

Note that aggregation invariance does not allow for different weights of decision makers. However, if the weights are rational numbers, then $AI$ is equivalent to this more general requirement.

\cite{PerezMokotoff2016} introduced a weaker property called \emph{group-coherence for choice} where alternative $i$ should have the highest priority in each pairwise comparison
matrices.

\begin{definition} \label{Def32}
\emph{Opposite of a pairwise comparison matrix}:
Let $\mathbf{A} = \left[ a_{ij} \right] \in \mathcal{A}^{n \times n}$ be a pairwise comparison matrix. Its \emph{opposite} is the pairwise comparison matrix $\mathbf{A}^- \in \mathcal{A}^{n \times n}$ such that $a_{ij}^- = 1 / a_{ij} = a_{ji}$ for all $1 \leq i,j \leq n$.
\end{definition}

Taking the opposite is equivalent to reversing all preferences of the decision-maker, and transposing the original pairwise comparison matrix.

\begin{axiom} \label{Axiom3}
\emph{Inversion} ($INV$):
Let $\mathbf{A} \in \mathcal{A}^{n \times n}$ be a pairwise comparison matrix.
Ranking method $g: \mathcal{A}^{n \times n} \to \mathfrak{R}^n$ is \emph{invertible} if $i \succeq^g_\mathbf{A} j \iff i \preceq^g_{\mathbf{A}^-} j$ for all $1 \leq i,j \leq n$.
\end{axiom}

Inversion implies that a reversal of all preferences changes the ranking accordingly.
An equivalent version of $INV$ for weighting methods has been implicitly investigated in \citet{JohnsonBeineWang1979}, and introduced under the name \emph{scale inversion} in \citet{Barzilai1997}.
An analogous axiom is \emph{invariance under inversion of preferences} for inconsistency indices \citep{Brunelli2017}, which requires the inconsistency of an arbitrary pairwise comparison matrix and its opposite to be the same.

The three properties below are probably first presented here.

\begin{axiom} \label{Axiom4}
\emph{Rational scale invariance} ($RSI$):
Let $\mathbf{A},\mathbf{A}^{(\kappa)} \in \mathcal{A}^{n \times n}$ be two pairwise comparison matrices such that $a^{(\kappa)}_{ij} = a_{ij}^\kappa$ for all $1 \leq i,j \leq n$ and $\kappa \in \mathbb{Q}_+$ is a positive rational number.
Ranking method $g: \mathcal{A}^{n \times n} \to \mathfrak{R}^n$ is called \emph{rational scale invariant} if $i \succeq^g_{\mathbf{A}} j \iff i \succeq^g_{\mathbf{A}^{(\kappa)}} j$.
\end{axiom}

Rational scale invariance is an adaptation of power invariance \citep{Fichtner1984} for ranking methods: the ordering of the alternatives does not change if a different scale is used for pairwise comparisons.
For example, when only two verbal expressions, 'weakly preferred' and 'strongly preferred' are allowed, the ranking is required to be the same if these preferences are represented by values $2$ and $3$, or $4$ and $9$, respectively.
This property has been implicitly investigated in \citet{GenestLapointeDrury1993}.

$RSI$ demands the invariance only in the case of a positive rational exponent. Naturally, one can define it for all positive real numbers, but this weaker form will be enough for us.

\begin{axiom} \label{Axiom5}
\emph{Independence of irrelevant comparisons} ($IIC$):
Let $\mathbf{A},\mathbf{A}' \in \mathcal{A}^{n \times n}$ be two pairwise comparison matrices and $1 \leq i,j,k,\ell \leq n$ be four different alternatives such that $\mathbf{A}$ and $\mathbf{A}'$ are identical but $a'_{k \ell} \neq a_{k \ell}$ ($a_{\ell k}' \neq a_{\ell k}$).
Ranking method $g: \mathcal{A}^{n \times n} \to \mathfrak{R}^n$ is called \emph{independent of irrelevant comparisons} if $i \succeq^g_{\mathbf{A}} j \iff i \succeq^g_{\mathbf{A}'} j$.
\end{axiom}

$IIC$ implies that 'remote' comparisons -- not involving alternatives $i$ and $j$ -- do not affect the pairwise ranking of $i$ and $j$. It has a meaning if $n \geq 4$.
Analogous axioms are extensively used in social choice theory, for example, in Arrow's impossibility theorem \citep{Arrow1950}.

Sequential application of independence of irrelevant comparisons may lead to any pairwise comparison matrix $\bar{\mathbf{A}} \in \mathcal{A}^{n \times n}$, for which $\bar{a}_{gh} = a_{gh}$ if $\{ g,h \} \cap \{ i, j \} \neq \emptyset$, but all other elements are arbitrary.

\begin{axiom} \label{Axiom6}
\emph{Responsiveness} ($RES$):
Let $\mathbf{A},\mathbf{A}' \in \mathcal{A}^{n \times n}$ be two pairwise comparison matrices and $1 \leq i,j \leq n$ be two different alternatives such that $\mathbf{A}$ and $\mathbf{A}'$ are identical but $a'_{ij} > a_{ij}$ ($a_{ji}' < a_{ji}$). 
Ranking method $g: \mathcal{A}^{n \times n} \to \mathfrak{R}^n$ is called \emph{responsive} if $i \succeq^g_{\mathbf{A}} j \Rightarrow i \succ^g_{\mathbf{A}'} j$.
\end{axiom}

Responsiveness is a natural monotonicity condition, similar to \emph{positive responsiveness} \citep{vandenBrinkGilles2009} and \emph{positive responsiveness to the beating relation} \citep{Gonzalez-DiazHendrickxLohmann2013}: if alternative $i$ is ranked at least as high as alternative $j$, then it should be ranked strictly higher when their comparison $a_{ij}$ changes in favour of alternative $i$.
An analogous axiom \emph{monotonicity on single comparisons} is used for inconsistency indices by \citet{BrunelliFedrizzi2015}, where the authors provide a further discussion of its origin.

To conclude, all of our six axioms have a parallel version in different topics such as social choice theory or measurement of inconsistency.

\section{Implications among the axioms} \label{Sec4}

In this section, some implications among properties presented in Section~\ref{Sec3} will be revealed.

\begin{lemma} \label{Lemma41}
$ANO$ and $AI$ imply $INV$.
\end{lemma}

\begin{proof}
Let $g: \mathcal{A}^{n \times n} \to \mathfrak{R}^n$ be a ranking method satisfying $ANO$ and $AI$.
Assume to the contrary that there exist alternatives $i$ and $j$ with a pairwise comparison matrix $\mathbf{A}$ such that $i \succeq^g_{\mathbf{A}} j$ and $i \succ^g_{\mathbf{A}^-} j$.
Consider the aggregated pairwise comparison matrix $\mathbf{A} + \mathbf{A}^- = \mathbf{1}$. Anonymity implies $i \sim^g_{\mathbf{1}} j$, while aggregation invariance leads to $i \succ^g_{\mathbf{1}} j$, a contradiction.
See also \citet[Lemma~4.1]{Csato2017b}.
\end{proof}

\begin{remark} \label{Rem41}
$ANO$ and $INV$ do not imply $AI$.
\end{remark}

Remark~\ref{Rem41} is verified by a counterexample.

\begin{example} \label{Examp41}
Consider the ranking method based on arithmetic means: $g: \mathcal{A}^{n \times n} \to \mathfrak{R}^n$ such that $i \succeq^g_{\mathbf{A}} j$ if $\sum_{k=1}^n a_{ik} \geq \sum_{k=1}^n a_{jk}$.
It is anonymous and invertible, but not aggregation invariant as the following matrices show:
\[
\mathbf{A}^{(1)} = \left[
\begin{array}{K{1.5em} K{1.5em} K{1.5em}}
    1     & 4     & 4     \\
     1/4  & 1     & 9     \\
     1/4  &  1/9  & 1     \\
\end{array}
\right] \quad \text{and} \quad
\mathbf{A}^{(2)} = \left[
\begin{array}{K{1.5em} K{1.5em} K{1.5em}}
    1     &  1/4  & 4     \\
    4     & 1     & 1     \\
     1/4  &  1    & 1     \\
\end{array}
\right] \text{, therefore}
\]
\[
\mathbf{B} = \mathbf{A}^{(1)} \oplus \mathbf{A}^{(2)} = \left[
\begin{array}{K{1.5em} K{1.5em} K{1.5em}}
    1     & 1     & 4     \\
    1     & 1     & 3     \\
     1/4  &  1/3  & 1     \\
\end{array}
\right].
\]
Here $1 \prec^g_{\mathbf{A}^{(1)}} 2$ because of $\sum_{k=1}^n a_{1k}^{(1)} = 9 < 10.25 = \sum_{k=1}^n a_{2k}^{(1)}$ and $1 \prec^g_{\mathbf{A}^{(2)}} 2$ as $\sum_{k=1}^n a_{1k}^{(2)} = 5.25 < 6 = \sum_{k=1}^n a_{2k}^{(1)}$, but $1 \succ^g_{\mathbf{B}} 2$ since $\sum_{k=1}^n b_{1k} = 6 > 5 = \sum_{k=1}^n b_{2k}$.
\end{example}

\begin{remark} \label{Rem42}
$AI$ and $INV$ do not imply $ANO$.
\end{remark}

Remark~\ref{Rem42} is verified by a counterexample.

\begin{example} \label{Examp42}
Consider the ranking method based on the first column: $g: \mathcal{A}^{n \times n} \to \mathfrak{R}^n$ such that $i \succeq^g_{\mathbf{A}} j$ if $a_{i1} \geq a_{j1}$.
It is aggregation invariant and invertible, but not anonymous.
\end{example}

\begin{lemma} \label{Lemma42}
$ANO$ and $AI$ imply $RSI$.
\end{lemma}

\begin{proof}
Consider two pairwise comparison matrices $\mathbf{A},\mathbf{A}^{(\kappa)} \in \mathcal{A}^{n \times n}$ and a ranking method $g: \mathcal{A}^{n \times n} \to \mathfrak{R}^n$ with $i \succeq^g_\mathbf{A} j$. It can be assumed without loss of generality that $\kappa =  k / \ell$ and $0 < k \leq \ell$, $k, \ell \in \mathbb{Z}$.
Then $a_{ij}^\kappa = a_{ij}^{k / \ell}$ is the geometric mean of $k$ pieces of $a_{ij}$ and $\ell - k$ pieces of $1$ for all $1 \leq i,j \leq n$.
In other words, $\mathbf{A}^{(\kappa)} = \mathbf{A} \oplus \dots \oplus \mathbf{A} \oplus \mathbf{1} \oplus \dots \oplus \mathbf{1}$, where the number of $\mathbf{A}$-s is $k$ and the number of $\mathbf{1}$-s is $k - \ell$ in the aggregation. Since $i \succeq^g_{\mathbf{1}} j$ due to anonymity, $i \succeq^g_{\mathbf{A}^{(\kappa)}} j$ is implied by aggregation invariance, thus $g$ is rational scale invariant.
\end{proof}

Note that the proof of Lemma~\ref{Lemma42} does not work directly if the exponent is allowed to be irrational.

\begin{remark} \label{Rem43}
$ANO$ and $RSI$ do not imply $AI$.
\end{remark}

Remark~\ref{Rem43} is verified by a counterexample.

\begin{example} \label{Examp43}
Consider the ranking method based on the product of favourable comparisons: $g: \mathcal{A}^{n \times n} \to \mathfrak{R}^n$ such that $i \succeq^g_{\mathbf{A}} j$ if $\prod_{k=1, a_{ik} \geq 1}^n a_{ik} \geq \prod_{k=1, a_{jk} \geq 1}^n a_{jk}$.
It is anonymous and rational scale invariant, but not aggregation invariant as the following matrices show:
\[
\mathbf{A}^{(1)} = \left[
\begin{array}{K{1.5em} K{1.5em} K{1.5em}}
    1       & 2       &   1/9  \\
      1/2  & 1       & 1       \\
    9       & 1       & 1       \\
\end{array}
\right] \quad \text{and} \quad
\mathbf{A}^{(2)} = \left[
\begin{array}{K{1.5em} K{1.5em} K{1.5em}}
    1       &   1/8  & 9       \\
    8       & 1       & 1       \\
      1/9  & 1       & 1       \\
\end{array}
\right] \text{, therefore}
\]
\[
\mathbf{B} = \mathbf{A}^{(1)} \oplus \mathbf{A}^{(2)} = \left[
\begin{array}{K{1.5em} K{1.5em} K{1.5em}}
    1       &   1/2  & 1       \\
    2       & 1       & 1       \\
    1       & 1       & 1       \\
\end{array}
\right].
\]
Here $1 \succ^g_{\mathbf{A}^{(1)}} 2$ because $2 > 1$ and $1 \succ^g_{\mathbf{A}^{(2)}} 2$ as $9 > 8$, but $1 \prec^g_{\mathbf{B}} 2$ since $1 < 2$.
\end{example}

\begin{remark} \label{Rem44}
$AI$ and $RSI$ do not imply $ANO$.
\end{remark}

Remark~\ref{Rem44} is verified by a counterexample.

\begin{example} \label{Examp44}
Consider the ranking method based on the first column: $g: \mathcal{A}^{n \times n} \to \mathfrak{R}^n$ such that $i \succeq^g_{\mathbf{A}} j$ if $a_{i1} \geq a_{j1}$.
It is aggregation invariant and rational scale invariant, but not anonymous.
\end{example}

\begin{lemma} \label{Lemma43}
$ANO$ and $AI$ imply $IIC$.
\end{lemma}

\begin{proof}
Assume to the contrary, and let $\mathbf{A}, \mathbf{B} \in \mathcal{A}^{n \times n}$ be two pairwise comparison matrices and $1 \leq i,j,k,\ell \leq n$ be four different alternatives such that $\mathbf{A}$ and $\mathbf{B}$ are identical except for $b_{k \ell} \neq a_{k \ell}$, furthermore, $g: \mathcal{A}^{n \times n} \to \mathfrak{R}^n$ is a ranking method with $i \succeq^g_\mathbf{A} j$ but $i \prec^g_\mathbf{B} j$.

An anonymous and aggregation invariant ranking method is invertible according to Lemma~\ref{Lemma41}, hence $i \succ^g_{\mathbf{B}^-} j$.
Denote by $\sigma: N \rightarrow N$ the permutation $\sigma(i) = j$, $\sigma(j) = i$, and $\sigma(k) = k$ for all $k \neq i,j$. Anonymity leads to $i \succ^g_{\sigma(\mathbf{B})} j$, while $i \succeq^g_{\sigma(\mathbf{A})^-} j$ because of $ANO$ and $INV$.

Consider the pairwise comparison matrix $\mathbf{C} = \mathbf{A} \oplus \mathbf{B}^- \oplus \sigma(\mathbf{A})^- \oplus \sigma(\mathbf{B})$. Its elements $c_{gh}$ are as follows:
\begin{itemize}
\item
$\{ g, h \} \cap \{ i, j \} \neq \emptyset$: it can be assumed without loss of generality that $g = i$. Then $c_{ih} = \sqrt[4]{a_{ih} \cdot 1 / a_{ih} \cdot 1 / a_{jh} \cdot a_{jh}} = 1$ since $b_{ih} = a_{ih}$, $\left[ \sigma(a) \right]_{ih} = a_{jh}$, and $\left[ \sigma(b) \right]_{ih} = b_{jh} = a_{jh}$.
\item
$\{ g, h \} \cap \{ i, j \} = \emptyset$ and $|\{ g, h \} \cap \{ k, \ell \}| \leq 1$: $c_{gh} = \sqrt[4]{a_{gh} \cdot 1 / a_{gh} \cdot 1 / a_{gh} \cdot a_{gh}} = 1$ since $b_{gh} = a_{gh}$, $\left[ \sigma(a) \right]_{gh} = a_{gh}$, and $\left[ \sigma(b) \right]_{gh} = b_{gh} = a_{jh}$.
\item
$|\{ g, h \} \cap \{ k, \ell \}| = 2$: it can be assumed without loss of generality that $g = k$ and $h = \ell$. Then $c_{k \ell} = \sqrt[4]{a_{k \ell} \cdot 1 / b_{k \ell} \cdot 1 / a_{k \ell} \cdot b_{k \ell}} = 1$.
\end{itemize}
Consequently, $\mathbf{C} = \mathbf{1}$, hence anonymity implies $i \sim^g_{\mathbf{C}} j$. However, $i \succ^g_{\mathbf{C}} j$ from aggregation invariance, which is a contradiction.
\end{proof}

\begin{remark} \label{Rem45}
$ANO$ and $IIC$ do not imply $AI$.
\end{remark}

Remark~\ref{Rem45} is verified by a counterexample.

\begin{example} \label{Examp45}
Consider the ranking method based on arithmetic means: $g: \mathcal{A}^{n \times n} \to \mathfrak{R}^n$ such that $i \succeq^g_{\mathbf{A}} j$ if $\sum_{k=1}^n a_{ik} \geq \sum_{k=1}^n a_{jk}$.
It is anonymous and independent of irrelevant comparisons, but not aggregation invariant (see Example~\ref{Examp41}).
\end{example}

\begin{remark} \label{Rem46}
$AI$ and $IIC$ do not imply $ANO$.
\end{remark}

Remark~\ref{Rem46} is verified by a counterexample.

\begin{example} \label{Examp46}
Consider the ranking method based on the first column: $g: \mathcal{A}^{n \times n} \to \mathfrak{R}^n$ such that $i \succeq^g_{\mathbf{A}} j$ if $a_{i1} \geq a_{j1}$.
It is aggregation invariant and independent of irrelevant comparisons, but not anonymous.
\end{example}

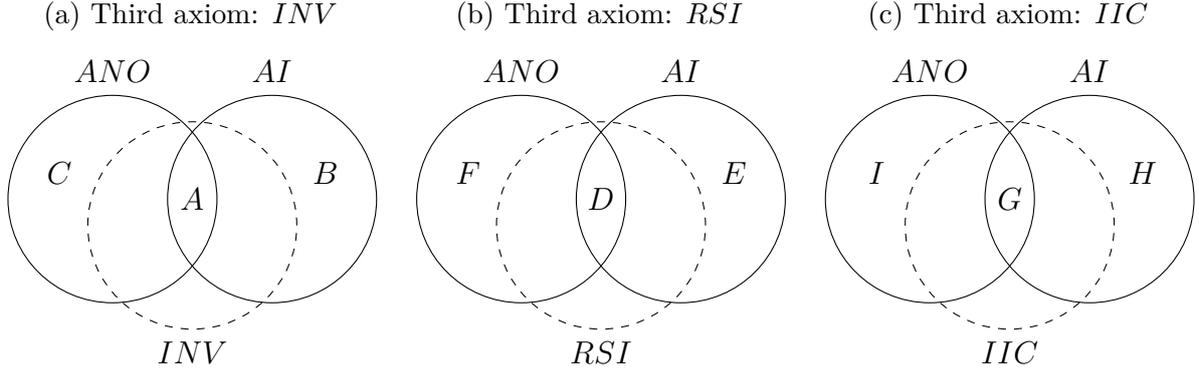
\begin{figure}[ht!]
\centering
\caption{Relations between $ANO$, $AI$, and a third axiom}
\label{Fig1}

\begin{subfigure}{.3275\textwidth}
\centering
\subcaption{Third axiom: $INV$}
\label{Fig1a}
\begin{tikzpicture}[scale=0.7, auto=center]
  \tikzset{venn circle/.style={draw,circle,minimum width=2.75cm}}
  \node [venn circle = red] (A) at (0,0) {};
  \node [venn circle = blue] (B) at (3,0) {};
  \node [venn circle = green,dashed] (C) at (1.5,-0.5) {};
  \node[above] at (0,2) {$ANO$};
  \node[above] at (3,2) {$AI$};
  \node[below] at (1.5,-2.5) {$INV$};
  \node at (barycentric cs:A=1/2,B=1/2) {$A$}; 
  \node at (4,0.5) {$B$};   
  \node at (-1,0.5) {$C$};   
\end{tikzpicture}
\end{subfigure}
\begin{subfigure}{.3275\textwidth}
\centering
\subcaption{Third axiom: $RSI$}
\label{Fig1b}
\begin{tikzpicture}[scale=0.7, auto=center]
  \tikzset{venn circle/.style={draw,circle,minimum width=2.75cm}}
  \node [venn circle = red] (A) at (0,0) {};
  \node [venn circle = blue] (B) at (3,0) {};
  \node [venn circle = green,dashed] (C) at (1.5,-0.5) {};
  \node[above] at (0,2) {$ANO$};
  \node[above] at (3,2) {$AI$};
  \node[below] at (1.5,-2.5) {$RSI$};
  \node at (barycentric cs:A=1/2,B=1/2) {$D$}; 
  \node at (4,0.5) {$E$};   
  \node at (-1,0.5) {$F$};   
\end{tikzpicture}
\end{subfigure}
\begin{subfigure}{.3275\textwidth}
\centering
\subcaption{Third axiom: $IIC$}
\label{Fig1c}
\begin{tikzpicture}[scale=0.7, auto=center]
  \tikzset{venn circle/.style={draw,circle,minimum width=2.75cm}}
  \node [venn circle = red] (A) at (0,0) {};
  \node [venn circle = blue] (B) at (3,0) {};
  \node [venn circle = green,dashed] (C) at (1.5,-0.5) {};
  \node[above] at (0,2) {$ANO$};
  \node[above] at (3,2) {$AI$};
  \node[below] at (1.5,-2.5) {$IIC$};
  \node at (barycentric cs:A=1/2,B=1/2) {$G$}; 
  \node at (4,0.5) {$H$};   
  \node at (-1,0.5) {$I$};   
\end{tikzpicture}
\end{subfigure}
\end{figure}

Main results of this section are summarized in Figure~\ref{Fig1}.
On Figure~\ref{Fig1a}, it can be seen that area $A$ is covered by the circle of axiom $INV$, in other words, $ANO$ and $AI$ imply $INV$ (Lemma~\ref{Lemma41}). Furthermore, the sets denoted by $B$ and $C$ are non-empty according to Remarks~\ref{Rem41} and \ref{Rem42}, respectively.
Analogously, since $ANO$ and $AI$ imply $RSI$ (Lemma~\ref{Lemma42}), area $D$ on Figure~\ref{Fig1b} is covered by the circle of axiom $RSI$, and the sets denoted by $E$ and $F$ are non-empty according to Remarks~\ref{Rem43} and \ref{Rem44}, respectively.
Finally, area $G$ on Figure~\ref{Fig1c} is covered by the circle of axiom $IIC$ due to Lemma~\ref{Lemma43}, and the sets denoted by $H$ and $I$ are non-empty according to Remarks~\ref{Rem45} and \ref{Rem46}, respectively.

\section{Analysis of two ranking methods} \label{Sec5}

In the following, we continue the investigation of the ranking methods presented in Section~\ref{Sec2}, with respect to Axioms~\ref{Axiom1}-\ref{Axiom6}.

\begin{proposition} \label{Prop51}
The eigenvector ranking method satisfies $ANO$ but violates $AI$, $INV$, $RSI$, and $IIC$.
\end{proposition}

\begin{proof}
See \citet[Lemma~4.2]{Csato2017b} for $ANO$.

Violation of $INV$ has been proved first probably in \citet{JohnsonBeineWang1979} and discussed in \citet[Lemma~4.3]{Csato2017b}. It implies the violation of $AI$ because of Lemma~\ref{Lemma41}.

For rational scale invariance, we use an example of \citet{GenestLapointeDrury1993}, adapted from \citet{Kendall1955}:
\[
\mathbf{A} = \left[
\begin{array}{K{1.5em} K{1.5em} K{1.5em} K{1.5em} K{1.5em} K{1.5em}}
    1     & 2     & 2     &  1/2  & 2     & 2 \\
     1/2  & 1     &  1/2  & 2     & 2     &  1/2 \\
     1/2  & 2     & 1     & 2     & 2     & 2 \\
    2     &  1/2  &  1/2  & 1     &  1/2  &  1/2 \\
     1/2  &  1/2  &  1/2  & 2     & 1     & 2 \\
     1/2  & 2     &  1/2  & 2     &  1/2  & 1 \\
\end{array}
\right] \quad \text{leads to} \quad
\mathbf{w}^{EM}(\mathbf{A}) = \left[
\begin{array}{c}
    0.2286 \\
    0.1430 \\
    0.2102 \\
    0.1321 \\
    0.1430 \\
    0.1430 \\
\end{array} 
\right] \text{, while }
\]
\[
\mathbf{A}^{(2)} = \left[
\begin{array}{K{1.5em} K{1.5em} K{1.5em} K{1.5em} K{1.5em} K{1.5em}}
    1     & 4     & 4     &  1/4  & 4     & 4 \\
     1/4  & 1     &  1/4  & 4     & 4     &  1/4 \\
     1/4  & 4     & 1     & 4     & 4     & 4 \\
    4     &  1/4  &  1/4  & 1     &  1/4  &  1/4 \\
     1/4  &  1/4  &  1/4  & 4     & 1     & 4 \\
     1/4  & 4     &  1/4  & 4     &  1/4  & 1 \\
\end{array}
\right] \quad \text{results in} \quad
\mathbf{w}^{EM}(\mathbf{A}^{(2)}) = \left[
\begin{array}{c}
    0.2640 \\
    0.1267 \\
    0.2261 \\
    0.1297 \\
    0.1267 \\
    0.1267 \\
\end{array} 
\right],
\]
therefore $2 \succ^{EM}_{\mathbf{A}} 4$ but $2 \prec^{EM}_{\mathbf{A}^{(2)}} 4$.

Breaking of $IIC$ can be verified by the following matrices:
\[
\mathbf{A} = \left[
\begin{array}{K{1.5em} K{1.5em} K{1.5em} K{1.5em}}
    1     & 1     & 1     & 3 \\
    1     & 1     & 2     & 1 \\
    1     & 1/2   & 1     & 1 \\
     1/3  & 1     & 1     & 1 \\
\end{array}
\right] \quad \text{leads to} \quad
\mathbf{w}^{EM}(\mathbf{A}) = \left[
\begin{array}{c}
    0.3254 \\
    0.2855 \\
    0.2034 \\
    0.1858 \\
\end{array} 
\right] \text{, while }
\]
\[
\mathbf{A}' = \left[
\begin{array}{K{1.5em} K{1.5em} K{1.5em} K{1.5em}}
    1     & 1     & 1     & 3 \\
    1     & 1     & 2     & 1 \\
    1     & 1/2   & 1     & 4 \\
     1/3  & 1     & 1/4  & 1 \\
\end{array}
\right] \quad \text{results in} \quad
\mathbf{w}^{EM}(\mathbf{A}') = \left[
\begin{array}{c}
    0.2880 \\
    0.2917 \\
    0.2855 \\
    0.1347 \\
\end{array} 
\right].
\]
Pairwise comparison matrices $\mathbf{A}$ and $\mathbf{A}'$ differ only in the comparison of alternatives $3$ and $4$, but $1 \succ^{EM}_{\mathbf{A}} 2$ and $1 \prec^{EM}_{\mathbf{A}'} 2$.
\end{proof}

On the basis of Proposition~\ref{Prop51}, eigenvector ranking method is placed somewhere in the region $C$ on Figure~\ref{Fig1a}, in the region $F$ on Figure~\ref{Fig1b}, and in the region $I$ on Figure~\ref{Fig1c}.

\begin{proposition} \label{Prop52}
The row geometric mean ranking method satisfies $ANO$, $AI$, $INV$, $RSI$, $IIC$, and $RES$.
\end{proposition}

\begin{proof}
The anonymity of row geometric mean ranking is obvious.

Aggregation invariance:
Take some pairwise comparison matrices $\mathbf{A}^{(1)},\mathbf{A}^{(2)}, \dots , \mathbf{A}^{(k)} \in \mathcal{A}^{n \times n}$ such that $i \succeq^{RGM}_{\mathbf{A}^{(\ell)}} j$, that is, $\prod_{m=1}^n a_{im}^{(\ell)} \geq \prod_{m=1}^n a_{jm}^{(\ell)}$ for all $1 \leq \ell \leq k$. It implies $\sqrt[k]{\prod_{\ell=1}^k \prod_{m=1}^n a_{im}^{(\ell)}} \geq \sqrt[k]{\prod_{\ell=1}^k \prod_{m=1}^n a_{jm}^{(\ell)}}$, which is equivalent to $i \succeq^{RGM}_{\mathbf{A}^{(1)} \oplus \mathbf{A}^{(2)} \oplus \dots \oplus \mathbf{A}^{(k)}} j$.
Furthermore, if $i \succ^{RGM}_{\mathbf{A}^{(\ell)}} j$, that is, $\prod_{m=1}^n a_{im}^{(\ell)} > \prod_{m=1}^n a_{jm}^{(\ell)}$ for at least one $1 \leq \ell \leq k$, then $\sqrt[k]{\prod_{\ell=1}^k \prod_{m=1}^n a_{im}^{(\ell)}} > \sqrt[k]{\prod_{\ell=1}^k \prod_{m=1}^n a_{jm}^{(\ell)}}$, so $i \succ^{RGM}_{\mathbf{A}^{(1)} \oplus \mathbf{A}^{(2)} \oplus \dots \oplus \mathbf{A}^{(k)}} j$.

Inversion, rational scale invariance and independence of irrelevant comparisons:
They immediately follow from $ANO$ and $AI$ according to Lemmata~\ref{Lemma41}, \ref{Lemma42}, and \ref{Lemma43}, respectively.

Responsiveness:
Let $\mathbf{A},\mathbf{A}' \in \mathcal{A}^{n \times n}$ be two pairwise comparison matrices and $1 \leq i,j \leq n$ be two different alternatives such that $\mathbf{A}$ and $\mathbf{A}'$ are identical but $a'_{ij} > a_{ij}$.
Assume that $i \succeq^{RGM}_{\mathbf{A}} j$, namely, $\prod_{k=1}^n a_{ik} \geq \prod_{k=1}^n a_{jk}$. Then $\prod_{k=1}^n a_{ik}' > \prod_{k=1}^n a_{ik} \geq \prod_{k=1}^n a_{jk} > \prod_{k=1}^n a_{jk}'$, therefore $i \succ^{RGM}_{\mathbf{A}'} j$.
\end{proof}

On the basis of Proposition~\ref{Prop52}, row geometric mean ranking method is placed somewhere in the region $A$ on Figure~\ref{Fig1a}, in the region $D$ on Figure~\ref{Fig1b}, and in the region $G$ on Figure~\ref{Fig1c}.

It is known that $EM$ and $RGM$ are equivalent if $n \leq 3$ \citep{CrawfordWilliams1985}.
Hence the counterexample for $IIC$ in Proposition~\ref{Prop51} is minimal with respect to the number of alternatives.
However, it remains to be seen whether the eigenvector ranking method satisfies $RSI$ for $n=4$ and $n=5$.

The eigenvector ranking method is not analysed with respect to responsiveness here.

\section{A characterization of the row geometric mean ranking method} \label{Sec6}

It has been presented in Section~\ref{Sec5} that the ranking induced by row geometric mean is compatible with the six properties introduced in Section~\ref{Sec3}. Lemmata~\ref{Lemma41}, \ref{Lemma42}, and \ref{Lemma43} have also revealed that $ANO$ and $AI$ are powerful axioms. According to our central result, they, together with $RES$, characterize this specific ordering. 

\begin{theorem} \label{Theo61}
The row geometric mean ranking method is the unique ranking method satisfying anonymity, aggregation invariance and responsiveness.
\end{theorem}

\begin{proof}
Row geometric mean ranking method satisfies $ANO$, $AI$ and $RES$ due to Proposition~\ref{Prop52}.

Take an arbitrary pairwise comparison matrix $\mathbf{A} \in \mathcal{A}^{n \times n}$ and an anonymous, aggregation invariant and responsive ranking method $g: \mathcal{A}^{n \times n} \to \mathfrak{R}^n$.
It can be assumed without loss of generality that $1 \succeq^{RGM}_{\mathbf{A}} 2$ because row geometric mean ranking satisfies $ANO$. It is enough to show that $1 \succeq^{g}_{\mathbf{A}} 2$ and $1 \succ^{g}_{\mathbf{A}} 2$ if $1 \succ^{RGM}_{\mathbf{A}} 2$.

Assume that $1 \sim^{RGM}_{\mathbf{A}} 2$, namely, $\prod_{j=1}^n a_{1j} = \prod_{j=1}^n a_{2j}$.
If $n=2$, then $\mathbf{A} = \mathbf{1}$, therefore anonymity provides that $1 \sim^g_{\mathbf{A}} 2$.

For $n \geq 3$, the proof is based on the following idea.
As a first step, some transformations will be made in order to get a pairwise comparison matrix where the ranking according to $g$ can be deduced from anonymity in the second step, while in the third part it will be proved that the pairwise ranking of alternatives $1$ and $2$ is not influenced by the previous transformations. Only the anonymity and aggregation invariance of $g$ will be used in this process.


\begin{enumerate}[label=\emph{\Roman*}.]
\item
Consider the pairwise comparison matrix $\mathbf{B} \in \mathcal{A}^{n \times n}$ such that $b_{1j} = a_{1j}$ and $b_{2j} = a_{2j}$ for all $1 \leq j \leq n$ but $b_{k \ell} = 1$ if $3 \leq k,\ell \leq n$. Note that $\mathbf{B} = \mathbf{A}$ if $n=3$.

Let $\sigma_{m}: N \to N$ be the permutation $\sigma_m(1) = 1$, $\sigma_m(2) = 2$, and $\sigma_m(k) = 3 + \left[ (m+k-3) \bmod (n-2) \right]$ for all $3 \leq k \leq n$ where $0 \leq m \leq n-3$.
For instance, $\sigma_1(3) = 4$, $\sigma_1(4) = 5$, and $\sigma_1(n) = 3$ if $n \geq 5$. Let $\sigma_m(\mathbf{B})$ be the pairwise comparison matrix obtained from $\mathbf{B}$ by permutation $\sigma_m$. It is clear that $\sigma_0(\mathbf{B}) = \mathbf{B}$.

Consider the pairwise comparison matrix $\mathbf{C} = \sigma_0(\mathbf{B}) \oplus \sigma_1(\mathbf{B}) \oplus \sigma_2(\mathbf{B}) \oplus \dots \oplus \sigma_{n-3}(\mathbf{B}) \in \mathcal{A}^{n \times n}$.
Its elements in the upper triangle -- which uniquely determine a pairwise comparison matrix due to its reciprocity -- are $c_{12} = a_{12}$, $c_{1k} = \sqrt[n-2]{\prod_{\ell=3}^n a_{1 \ell}}$ for all $3 \leq k \leq n$, $c_{2k} = \sqrt[n-2]{\prod_{\ell=3}^n a_{2 \ell}}$ for all $3 \leq k \leq n$, and $c_{k \ell} = 1$ for all $3 \leq k,\ell \leq n$.
Note that $\mathbf{C} = \mathbf{B}$ if $n=3$.

Define the pairwise comparison matrix $\mathbf{D} \in \mathcal{A}^{n \times n}$ such that $d_{12} = 1$, $d_{1k} = 1 / \left( \sqrt[n-2]{\prod_{j=1}^n a_{1j}} \right)$ for all $3 \leq k \leq n$, $d_{2k} = 1 / \left( \sqrt[n-2]{\prod_{j=1}^n a_{2j}} \right)$ for all $3 \leq k \leq n$, and $d_{k \ell} = 1$ for all $3 \leq k,\ell \leq n$.

Consider the pairwise comparison matrix $\mathbf{E} = \mathbf{C} \oplus \mathbf{D}$. Its elements are $e_{12} = \sqrt{a_{12}} = \alpha$, $e_{1k} = \sqrt[n-2]{1 / \alpha}$ for all $3 \leq k \leq n$, $e_{2k} = \sqrt[n-2]{\alpha}$ for all $3 \leq k \leq n$, and $e_{k \ell} = 1$ for all $3 \leq k,\ell \leq n$ as the geometric means of its row elements are ones.

\item
It is shown that $1 \sim^g_{\mathbf{E}} 2 \sim^g_{\mathbf{E}} \dots \sim^g_{\mathbf{E}} n$.
Anonymity implies $3 \sim^g_{\mathbf{E}} 4 \sim^g_{\mathbf{E}} \dots \sim^g_{\mathbf{E}} n$.

Let $\sigma_{1,2}: N \to N$ be the permutation $\sigma_{1,2}(1) = 2$, $\sigma_{1,2}(2) = 1$, and $\sigma_{1,2}(k) = k$ for all $3 \leq k \leq n$. Let $\sigma_{1,2}(\mathbf{E})$ be the pairwise comparison matrix obtained from $\mathbf{E}$ by the permutation $\sigma_{1,2}$.
Note that $\sigma_{1,2}(\mathbf{E}) = \mathbf{E}^-$. Ranking method $g$ is anonymous and aggregation invariant, so it satisfies inversion according to Lemma~\ref{Lemma41}.
If $1 \succ^g_{\mathbf{E}} 3$ and $2 \succ^g_{\mathbf{E}} 3$, then $ANO$ implies $2 \succ^g_{\sigma_{1,2}(\mathbf{E})} 3$, but $INV$ results in $2 \prec^g_{\mathbf{E}^-} 3$, a contradiction.
If $1 \prec^g_{\mathbf{E}} 3$ and $2 \prec^g_{\mathbf{E}} 3$, then $ANO$ implies $2 \prec^g_{\sigma_{1,2}(\mathbf{E})} 3$, but $INV$ results in $2 \succ^g_{\mathbf{E}^-} 3$, a contradiction.

Due to the anonymity of the ranking method $g$, it can be supposed without loss of generality that $1 \succ^g_{\mathbf{E}} \left( 3 \sim^g_{\mathbf{E}} 4 \sim^g_{\mathbf{E}} \dots \sim^g_{\mathbf{E}} n \right) \succ^g_{\mathbf{E}} 2$.
Let $\sigma_{2,m}: N \to N$ be the permutation $\sigma_{2,m}(1) = 1$, $\sigma_{2,m}(2) = m$, $\sigma_{2,m}(m) = 2$, and $\sigma_{2,m}(\ell) = \ell$ for all $\ell \neq m$, $3 \leq \ell \leq n$ where $3 \leq m \leq n$. Let $\sigma_{2,m}(\mathbf{E})$ be the pairwise comparison matrix obtained from $\mathbf{E}$ by the permutation $\sigma_{2,m}$.

It can be checked that $\left[ \sigma_{1,2}(e) \right]_{ij} = \left[ \sigma_{2,3}(e) \right]_{ij} \left[ \sigma_{2,4}(e) \right]_{ij} \cdots \left[ \sigma_{2,n}(e) \right]_{ij}$ for all $1 \leq i,j \leq n$, in other words, $\sigma_{1,2}(\mathbf{E})^{\left( 1/(n-2) \right)} = \sigma_{2,3}(\mathbf{E}) \oplus \sigma_{2,4}(\mathbf{E}) \oplus \dots \oplus \sigma_{2,n}(\mathbf{E})$. 
Anonymity implies $1 \succ^g_{\sigma_{2,m}(\mathbf{E})} 2$ for all $3 \leq m \leq n$, therefore aggregation invariance leads to $1 \succ^g_{\sigma_{2,3}(\mathbf{E}) \oplus \sigma_{2,4}(\mathbf{E}) \oplus \dots \oplus \sigma_{2,n}(\mathbf{E})} 2$, and rational scale invariance (an immediate consequence of $ANO$ and $AI$ according to Lemma~\ref{Lemma42}) results in $1 \succ^g_{\sigma_{1,2}(\mathbf{E})} 2$.
But $INV$ and $1 \succ^g_{\mathbf{E}} 2$ also leads to $1 \prec^g_{\mathbf{E}^-} 2$, which is a contradiction.

To summarize, we have derived $1 \sim^g_{\mathbf{E}} 2 \sim^g_{\mathbf{E}} 3 \sim^g_{\mathbf{E}} \dots \sim^g_{\mathbf{E}} n$.

\item
Anonymity implies $1 \sim^g_{\mathbf{D}} 2$ since $d_{12} = 1$ and $d_{1k} = d_{2k}$ for all $3 \leq k \leq n$, which means $1 \sim^g_{\mathbf{C}} 2$ because $1 \sim^g_{\mathbf{E}} 2$, $\mathbf{E} = \mathbf{C} \oplus \mathbf{D}$ and $g$ is aggregation invariant. Furthermore, permutation $\sigma_m$, used in the definition of pairwise comparison matrix $\mathbf{B}$, does not affect alternatives $1$ and $2$, therefore $1 \sim^g_{\mathbf{B}} 2$. $ANO$ and $AI$ also imply independence of irrelevant comparisons (see Lemma~\ref{Lemma43}), hence $1 \sim^g_{\mathbf{A}} 2$.
\end{enumerate}

We have verified up to this point that $1 \sim^{RGM}_{\mathbf{A}} 2$ implies $1 \sim^g_{\mathbf{A}} 2$.
If $1 \succ^{RGM}_{\mathbf{A}} 2$, namely, $\prod_{j=1}^n a_{1j} > \prod_{j=1}^n a_{2j}$, then consider the pairwise comparison matrix $\mathbf{A}' \in \mathcal{A}^{n \times n}$ where $a_{k \ell}' = a_{k \ell}$ for all $1 \leq k,\ell \leq n$ except for $a_{12}' = a_{12} \sqrt{\prod_{j=1}^n a_{2j} / \prod_{j=1}^n a_{1j}}$ as well as $a_{21}' = a_{21} \sqrt{\prod_{j=1}^n a_{1j} / \prod_{j=1}^n a_{2j}}$ in order to preserve reciprocity.
Therefore
\[
\prod_{j=1}^n a_{1j}' = \left( \prod_{j=1}^n a_{1j} \right) \sqrt{\frac{\prod_{j=1}^n a_{2j}}{\prod_{j=1}^n a_{1j}}} = \left( \prod_{j=1}^n a_{2j} \right) \sqrt{\frac{\prod_{j=1}^n a_{1j}}{\prod_{j=1}^n a_{2j}}} = \prod_{j=1}^n a_{2j}',
\]
and it has been proved above that $1 \sim^g_{\mathbf{A}'} 2$. So $1 \succ^g_{\mathbf{A}} 2$ due to the responsiveness of $g$.
\end{proof}

\begin{example} \label{Examp61}
As an illustration of the proof of Theorem~\ref{Theo61}, it is worth to consider the pairwise comparison matrices used in the derivations, which are as follows for $n=4$:
\[
\mathbf{A} = \left[
\begin{array}{cccc}
    1			& a_{12}	& a_{13}	& a_{14} \\
    1/a_{12}	& 1			& a_{23}	& a_{24} \\
    1/a_{13}	& 1/a_{23}	& 1     	& a_{34} \\
    1/a_{14}	& 1/a_{24}	& 1/a_{34}	& 1 \\
\end{array}
\right], \quad
\mathbf{B} = \sigma_0(\mathbf{B}) = \left[
\begin{array}{cccc}
    1			& a_{12}	& a_{13}	& a_{14} \\
    1/a_{12}	& 1			& a_{23}	& a_{24} \\
    1/a_{13}	& 1/a_{23}	& 1     	& 1 \\
    1/a_{14}	& 1/a_{24}	& 1			& 1 \\
\end{array}
\right], \quad
\]
\[
\sigma_1(\mathbf{B}) = \left[
\begin{array}{cccc}
    1			& a_{12}	& a_{14}	& a_{13} \\
    1/a_{12}	& 1			& a_{24}	& a_{23} \\
    1/a_{14}	& 1/a_{24}	& 1     	& 1 \\
    1/a_{13}	& 1/a_{23}	& 1			& 1 \\
\end{array}
\right],
\]
\[
\mathbf{C} = \sigma_0(\mathbf{B}) \oplus \sigma_1(\mathbf{B}) = \left[
\begin{array}{cccc}
    1						& a_{12}					& \sqrt{a_{13} a_{14}}	& \sqrt{a_{13} a_{14}} \\
    1/a_{12}				& 1							& \sqrt{a_{23} a_{24}}	& \sqrt{a_{23} a_{24}} \\
    1/ \sqrt{a_{13} a_{14}}	& 1/ \sqrt{a_{23} a_{24}}	& 1     				& 1 \\
    1/ \sqrt{a_{13} a_{14}}	& 1/ \sqrt{a_{23} a_{24}}	& 1						& 1 \\
\end{array}
\right], \quad
\]
\[
\mathbf{D} = \left[
\begin{array}{cccc}
    1	& 1	& 1 / \sqrt{a_{12} a_{13} a_{14}}	& 1 / \sqrt{a_{12} a_{13} a_{14}} \\
    1	& 1	& \sqrt{a_{12}} / \sqrt{a_{23} a_{24}}	& \sqrt{a_{12}} / \sqrt{a_{23} a_{24}} \\
    \sqrt{a_{12} a_{13} a_{14}}	& \sqrt{a_{23} a_{24}} / \sqrt{a_{12}}	& 1		& 1 \\
    \sqrt{a_{12} a_{13} a_{14}}	& \sqrt{a_{23} a_{24}} / \sqrt{a_{12}}	& 1		& 1 \\
\end{array}
\right].
\]
Recall that $\alpha = \sqrt{a_{12}}$, hence
\[
\mathbf{E} = \mathbf{C} \oplus \mathbf{D} \left[
\begin{array}{cccc}
    1					& \alpha			& \frac{1}{\sqrt{\alpha}}	& \frac{1}{\sqrt{\alpha}} \\
    \frac{1}{\alpha}	& 1					& \sqrt{\alpha}				& \sqrt{\alpha} \\
    \sqrt{\alpha}		& \frac{1}{\sqrt{\alpha}}	& 1     			& 1 \\
    \sqrt{\alpha}		& \frac{1}{\sqrt{\alpha}}	& 1					& 1 \\
\end{array}
\right], \quad
\sigma_{1,2}(\mathbf{E}) = \mathbf{E}^- = \left[
\begin{array}{cccc}
    1				& \frac{1}{\alpha}		& \sqrt{\alpha}				& \sqrt{\alpha} \\
    \alpha			& 1						& \frac{1}{\sqrt{\alpha}}	& \frac{1}{\sqrt{\alpha}} \\
    \frac{1}{\sqrt{\alpha}}	& \sqrt{\alpha}	& 1     			& 1 \\
    \frac{1}{\sqrt{\alpha}}	& \sqrt{\alpha}	& 1					& 1 \\
\end{array}
\right], \quad
\]
\[
\sigma_{2,3}(\mathbf{E}) = \left[
\begin{array}{cccc}
    1					& \frac{1}{\sqrt{\alpha}}	& \alpha	& \frac{1}{\sqrt{\alpha}} \\
    \sqrt{\alpha}		& 1							& \frac{1}{\sqrt{\alpha}}	& 1 \\
    \frac{1}{\alpha}	& \sqrt{\alpha}				& 1     	& \sqrt{\alpha} \\
    \sqrt{\alpha}		& 1							& \frac{1}{\sqrt{\alpha}}	& 1 \\
\end{array}
\right], \quad
\sigma_{2,4}(\mathbf{E}) = \mathbf{E}^- = \left[
\begin{array}{cccc}
    1					& \frac{1}{\sqrt{\alpha}}	& \frac{1}{\sqrt{\alpha}}	& \alpha \\
    \sqrt{\alpha}		& 1							& 1	& \frac{1}{\sqrt{\alpha}} \\
    \sqrt{\alpha}		& 1							& 1 & \frac{1}{\sqrt{\alpha}} \\
    \frac{1}{\alpha}	& \sqrt{\alpha}				& \sqrt{\alpha}				& 1 \\
\end{array}
\right].
\]
\end{example}

All three properties used in the proof of Theorem~\ref{Theo61} are necessary according to the following result.

\begin{proposition} \label{Prop61}
$ANO$, $AI$, and $RES$ are logically independent axioms.
\end{proposition}

\begin{proof}
It is shown that there exist ranking methods, which satisfy exactly two properties from this set of three, but differ from the row geometric mean ranking method (and therefore they are guaranteed to violate the third axiom):
\begin{enumerate}[label=\fbox{\arabic*}]
\item
$ANO$ and $AI$: flat ranking method, $g: \mathcal{A}^{n \times n} \to \mathfrak{R}^n$ such that $i \sim^g_{\mathbf{A}} j$ for all alternatives $i,j \in N$ and any pairwise comparison matrix $\mathbf{A} \in \mathcal{A}^{n \times n}$;
\item
$ANO$ and $RES$: row arithmetic mean ranking method, $g: \mathcal{A}^{n \times n} \to \mathfrak{R}^n$ such that $i \succeq^g_{\mathbf{A}} j$ for all alternatives $i,j \in N$ and for any pairwise comparison matrix $\mathbf{A} \in \mathcal{A}^{n \times n}$  if $\sum_{k=1}^n a_{ik} \geq \sum_{k=1}^n a_{jk}$;
\item
$AI$ and $RES$: a ranking method based on indices, $g: \mathcal{A}^{n \times n} \to \mathfrak{R}^n$ such that $i \succ^g_{\mathbf{A}} j$ for all alternatives $i,j \in N$ and for any pairwise comparison matrix $\mathbf{A} \in \mathcal{A}^{n \times n}$ if $i < j$.
\end{enumerate}
\end{proof}

\begin{figure}[ht!]
\centering
\caption{Relations between $ANO$, $AI$, and $RES$}
\label{Fig2}
\begin{tikzpicture}[scale=2, >=triangle 45]
  \draw (-1,0) -- (0,1) -- (1,0) -- (0,-1) -- (-1,0);
  \draw (0,0) -- (1,1) -- (2,0) -- (1,-1) -- (0,0);
  \draw (0.5-0.707106781186548,-0.5) -- (0.5+0.707106781186548,-0.5) -- (0.5+0.707106781186548,-0.5-2*0.707106781186548) -- (0.5-0.707106781186548,-0.5-2*0.707106781186548) -- (0.5-0.707106781186548,-0.5);
  \node at (-0.75,0.75) {$ANO$};
  \node at (1.75,0.75) {$AI$};
  \node[below] at (0.5,-2) {$RES$};
  \node at (0.5,-0.5) [circle,fill,inner sep=2pt]{};
  \node at (0.5,0) {$J$};
  \node at (0,-0.7) {$K$};
  \node at (1,-0.7) {$L$};
  \node (n) at (2.5,0.5) {$RGM$};
  \draw [->] (n) -- (0.55,-0.485);
\end{tikzpicture}
\end{figure}
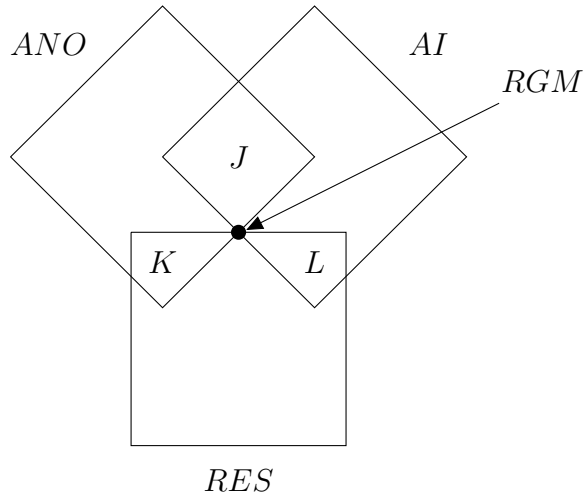

Figure~\ref{Fig2} summarizes our findings from Theorem~\ref{Theo61} and Proposition~\ref{Prop61}. First, the three axioms meet at a unique point, denoted by the dot, since there is a unique ranking method, the one induced by row geometric mean, that satisfies all of them. Second, the intersection of any two properties, denoted by the labels $J$, $K$, and $L$, is non-empty.

It is clear that $ANO$, $AI$ and \emph{negative responsiveness} (requiring the implication $i \preceq^g_{\mathbf{A}} j \Rightarrow i \prec^g_{\mathbf{A}'} j$ under the conditions of Axiom~\ref{Axiom6}) are also independent and uniquely determine the ordering opposite to the row geometric mean ranking. Naturally, this observation has only a technical sense.

\section{Discussion} \label{Sec7}

We have examined the problem of extracting a ranking of alternatives from a reciprocal pairwise comparison ratio matrix. A characterization of the ordering induced by row geometric means has been presented, which shows that a unique ranking can be derived by requiring anonymity, aggregation invariance, and responsiveness. It is a solid argument in favour of this particular method.

We do not suggest to accept the three axioms immediately. However, $ANO$ and $RES$ seem to be difficult to debate, whereas $AI$ follows from a well-known result of synthesizing ratio judgements \citep{AczelSaaty1983}. Perhaps it is not only a coincidence that row geometric mean has a number of other favourable properties (see, e.g. \citet{BarzilaiCookGolany1987, Barzilai1997, Dijkstra2013, Csato2015a, LundySirajGreco2017, Csato2018c}).

There are some obvious topics for further research.
It is worth to consider whether certain axioms (especially aggregation invariance) can be substituted in our main theorem.
Responsiveness of the eigenvector ranking method has been not discussed here.
Several other methods can be analysed with respect to these axioms.
Finally, an extension to the incomplete case, when some elements of the pairwise comparison matrix may be missing, deserves a thorough investigation.
Row geometric mean method has been defined on this more general domain by \citet{BozokiFulopRonyai2010} on the basis of optimization problem \eqref{Eq_LLSM}, without affecting at least one desirable property of the procedure \citep{BozokiTsyganok2017}.

\section*{Acknowledgements}
\addcontentsline{toc}{section}{Acknowledgements}
\noindent
We are grateful to \emph{S\'andor Boz\'oki} and \emph{Matteo Brunelli} for useful advice.
We thank two anonymous referees for beneficial remarks and suggestions. \\
The research was supported by OTKA grant K 111797 and by the MTA Premium Post Doctorate Research Program.


\end{document}